\documentclass{amsart}
\usepackage{fullpage}
\usepackage{amsfonts}
\usepackage{latexsym}
\usepackage{amscd}
\usepackage{amsmath}
\usepackage{amssymb}
\usepackage{graphicx}
\usepackage[mathscr]{eucal}
\usepackage{xypic}

\newtheorem{theorem}{Theorem}[section]
\newtheorem{lemma}[theorem]{Lemma}
\newtheorem{corollary}[theorem]{Corollary}

\newtheorem{definition}[theorem]{Definition}
\newtheorem{example}[theorem]{Example}

\theoremstyle{remark}

\numberwithin{equation}{section}

\newcommand{\N}{\mathbb{N}}

\begin{document}

\title[The Ellis semigroup of a nonautonomous discrete dynamical system]{The Ellis semigroup of a nonautonomous discrete dynamical system}

\author{S. Garc\'{\i}a-Ferreira}
\address{Centro de Ciencias Matem\'aticas,   Universidad Nacional
Aut\'onoma de M\'exico, Campus Morelia, Apartado Postal 61-3, Santa Mar\'{\i}a,
58089, Morelia, Michoac\'an, M\'exico}
\email{sgarcia@matmor.unam.mx}

\author{M. Sanchis}
\address{Institut de Matem\`atiques i Aplicacions de
Castell\'o (IMAC), Universitat Jaume I,
 Campus Riu Sec, 12071-Castell\'{o}, Spain }
\email{sanchis@mat.uji.es}

\thanks{Research of the first-named author was supported
by  CONACYT grant no. 81368-F and PAPIIT grant no. IN-101911.
Hospitality and financial support received from the \emph{Department
of Mathematics of Jaume I University (Spain)} where this research
was essentially performed are gratefully acknowledged. The research
of the second-named author was supported by the Spanish Ministry of
Science and Education (Grant number MTM2011-23118), and by Bancaixa
(Projecte P1$\cdot$1B2011-30.}

\subjclass[2010]{Primary 54G20, 54D80, 22A99: secondary 54H11}

\date{}

\dedicatory{}

\keywords{free ultrafilter, discrete dynamical system, nonautonomous discrete dynamical system, Ellis semigroup,
compact metric space, $p$-limit point}

\begin{abstract}  We introduce the  {\it Ellis semigroup} of a nonautonomous discrete dynamical system
$(X,f_{1,\infty})$ when $X$ is a metric compact space. The underlying set of this semigroup  is the pointwise closure of
$\{f\sp{n}_1 \, |\, n\in \mathbb{N}\}$ in the space $X\sp{X}$.
 By using the convergence  of a sequence of points with respect to an ultrafilter it is possible to give a precise description of the semigroup and its operation.
This notion  extends the classical Ellis semigroup of a discrete dynamical system. We show several properties that connect  this semigroup and the  topological properties of the nonautonomous discrete dynamical system.
\end{abstract}

\maketitle

\section{Introduction and preliminaries}

 A general form of a nonautonomous  difference equation is the following:
Given a compact metric space $(X,d)$ and a sequence of continuous function $(f_n: X \to X)_{n \in \N }$,
for each $x  \in X$ we set
$$
 \begin{cases} x_0 = x,  \\  ¿
x_{n+1} = f_n(x_n).
 \end{cases}
$$
These kind of nonautonomous  difference equations have been considered by several mathematicians (see for instance, among others,
\cite{shi}, \cite{zhu}). The most classical examples are when $X= [0,1]$ is the unit interval, and $d$ is the usual euclidean metric. After seeing the definition of a nonautonomous discrete dynamical system, we observe that  the orbit of a point forms a solution of  a nonautonomous  difference equation.

 \medskip

 In this paper, we shall apply the notion of convergence with respect to an ultrafilter on the natural numbers $\N$ to the study of the
 nonautonomous discrete dynamical system. This kind of convergence has been a very powerful tool in the  study of several topological aspects of  a discrete dynamical system. In particular, it is very useful to handle the elements  of the Ellis semigroup  as we can seen in the papers  \cite{Bla},  \cite{gf11}, \cite{gs} and \cite{gs2}.  The aim of this paper is twofold. First we introduce and study the Ellis semigroup of a nonautonomous discrete dynamical system. Second, we point out how this convergent theory  can be used in the context of nonautonomous dynamical systems. So far as we know, it is the first time that this kind of convergence has been applied in the framework of nonautonomous discrete dynamical systems.

\medskip

Our  topological spaces will be compact and metric.
Usually, $d$ will denote the metric of a metric space $X$ and $B(x,\epsilon)$ the ball with center $x$ and radius $\epsilon > 0$.
The set of neighborhoods of a point $x \in X$ will be denoted by $\mathcal{N}(x)$.  If $f: X \to Y$ is a continuous function between Tychonoff spaces, then
$\overline{f}: \beta(X) \to \beta(Y)$ will stand for the Stone
extension of $f$. A subbasic open subset of a power space $X^I$ is denoted by $[x,U] = \{ f \in X^I : f(i) \in V \}$ where $i \in I$ and $U$ is a nonempty open subset of $X$. The Stone-\v Cech compactification
$\beta(\mathbb{N})$ of the natural numbers $\mathbb{N}$ with the
discrete topology will be identified with the set of all
ultrafilters on $\mathbb{N}$, and its remainder $\mathbb{N}^* =
\beta(\mathbb{N}) \setminus \mathbb{N}$ with the set of all free
ultrafilters on $\mathbb{N}$. Notice  that each $n \in \N$ can be identified
with the fixed ultrafilter $\{ A \subseteq \N : n \in A \}$. If $A \subseteq \mathbb{N}$, then
$\hat{A} = cl_{\beta(\mathbb{N})}A = \{ p \in \beta(\mathbb{N}) : A
\in p \}$ is a basic clopen subset of $\beta(\mathbb{N})$, and $A^*
= \hat{A} \setminus A = \{ p \in \mathbb{N}^* : A \in p \}$ is a
basic clopen subset of $\mathbb{N}^*$.

\medskip

Recall that  a {\it discrete dynamical system} is a pair $(X,f)$ where $X$ is
a compact metric space and $f\colon X\to X$ is a continuous function.
A {\it nonautonomous discrete dynamical system} is  a pair $(X,f_{1,\infty})$  where
$X$ is a compact metric space and  $f_{1,\infty}$ is a sequence of continuous functions
$(f_n: X \to X)_{n \in \N }$. The nonautonomous discrete dynamical systems were introduced by  S.  Kolyada and L. Snoha in the article \cite{ks}. The  paper \cite{bali} describes some recent developments on the theory of nonautonomous discrete dynamical systems (see also \cite{ca}, \cite{ca11}, \cite{dth}, \cite{dth15}, \cite{cl} and \cite{shi}).
Given  $n \in \N$, the $n$-{\it iterate} of a nonautonomous discrete dynamical system  $(X,f_{1,\infty})$  is the composition
 $$
 f_1^n := f_n \circ f_{n-1}  \circ....... \circ f_2 \circ f_1.
 $$
The symbol $f^0_1$ will stand for the identity map.
A  discrete dyna\-mi\-cal system $(X,f)$ coincides with  the nonautonomous discrete dynamical system
$(X,f_{1,\infty})$ where $f_n = f$ for each $n \in \N$. If $(X,f_{1,\infty})$ is
a nonautonomous discrete dynamical system, then the {\it orbit} of a point $x \in X$ is the
set
$$
\mathcal{O}_{f_{1,\infty}}(x) := \{x, f_1^1(x), f_1^2(x),....., f_1^n(x),.....\}.
$$

\medskip

The {\it Ellis semigroup} of a discrete dynamical system $(X,f)$, denoted by $E(X,f)$,  is the closure of $\{ f^n : n \in \N \}$ inside of  the compact space
$X^X$. It well-known that $E(X,f)$  is a compact semigroup whose operation is the composition of functions (see, for instance,  \cite{gs}, \cite{gs2} and \cite{glasner}).  The Ellis semigroup of a discrete dynamical system was introduced by R. Ellis in \cite{ell} and has played a very important role in topological
dynamics.  The article \cite{glasner} offers an excellent survey concerning applications of the Ellis semigroup.

\medskip

The second section is devoted to define and describe  the Ellis semigroup for a nonautonomous discrete dynamical system.
Several basic properties of the Ellis semigroup are proved in this section.
In the third section, we study the Ellis semigroup of  those nonautonomous discrete dynamical systems $(X,f_{1,\infty})$ for which
the sequence $(f_n)_{n \in \N}$ converges uniformly to a function $\phi: X \to X$.

\section{The Ellis semigroup}

We start with the description  of the underlying space of the Ellis semigroup of a nonautonomous discrete dynamical system.

\begin{definition} The {\it Ellis semigroup} of a nonautonomous discrete dynamical system
$(X,f_{1,\infty})$ is the pointwise closure of the set
$\{f_1\sp{n}\, |\, n\in \mathbb{N}\}$ inside of the compact space $X\sp{X}$.
This compact space will be denoted by $E(X,f_{1,\infty})$.
\end{definition}

Our next task is the definition of the  semigroup operation of $E(X,f_{1,\infty})$ which justifies the name Ellis semigroup and we shall also explain why  this operation extends the original operation for the case of a discrete dynamical system.  To have this done  we shall follow some ideas from  the papers  \cite{gs} and \cite{gs2}. One of such ideas  is the use of   the $p$-limit point of a sequence of points, where $p \in  \N^*$:

\medskip

Let $X$ be a space and  $p \in \mathbb{N}^*$. A point $x \in X$ is
said to be the  $p$-{\it limit point} of a sequence $(x_n)_{n \in
\mathbb{N}}$ in $X$ (in symbols, $x = p-\lim_{n \to \infty}x_n$) if for
every neighborhood $V$ of $x$, $\{ n \in \mathbb{N} : x_n \in V \}
\in p$.  The notion of $p$-limit point was
introduced by several mathematicians in distinct contexts, for instance we can mention R.
A. Bernstein \cite{Be} and H. Furstenberg \cite[p. 179]{fu}. We remark that each
sequence of a compact space always has a $p$-limit point for every
$p \in \mathbb{N}^*$. Indeed, if $f: \mathbb{N} \to X$ is an arbitrary
function and $X$ is compact and Hausdorff, then  $\overline{f}(p)$  is precisely the $p$-limit point of the sequence
$(f(n))_{n \in \mathbb{N}}$ (this fact is explained in \cite{gon}).  Thus, the $p$-limit point of a sequence in
a compact space not only exists but is unique. Besides, the $p$-limit points are preserved under continuous functions.
We also remark that a point $x \in X$ is an accumulation point of
a countable set $\{ x_n :\, n \in \mathbb{N} \}$ of $X$ iff there is $p \in
\mathbb{N}^*$ such that $x = p - \lim_{n \to \infty} x_n$.

\begin{definition}\label{2.1}
Let $(X,f_{1,\infty})$ be a nonautonomous discrete dynamical system. For each $p \in\N^*$, we define
the function $f_1^p: X \to X$ by
$$
f_1^p(x)= p-\lim_{n \to \infty}f_1^n(x),
$$
for every $x \in X$.  This function $f_1^p$ is called the $p$-iterate of the  nonautonomous discrete dynamical system $(X,f_{1,\infty})$, for each $p \in \N^*$.
For a discrete dynamical system $E(X,f)$, we simply write $f^p$ for each $p \in \beta(\N)$.
\end{definition}

Notice that all iterates  $f_1^n$'s are always continuous. However, the functions $f_1^p$'s
could be all discontinuous: For instance,
let $X = [0,1]$ and  consider any sequence of continuous functions $(f_n: X \to X)_{n \in \N}$  such that $f_n(x) \to 0$ for each $x \in (0,1)$,
$f_n(0) = 0$ and $f_n(1) = 1$ and $f_{n+1} < f_{n}$ for each $n \in \N$.
If $p \in \mathbb{N}^*$, then we have that $f^p(x) = 0$ for every $x < 1$, and
$f^p(1) = 1$. It then follows that $f^p$ is discontinuous at $1$,
for all $p \in \mathbb{N}^*$.
In Theorem \ref{itecon}, we shall give some condition in order that the function $f_1^p$ be continuous, for every $p \in \N^*$.

\medskip

The following theorem shows that the elements of the Ellis semigroup of a nonautonomous discrete dynamical system are precisely the
$p$-iterations of the sequence $f_{1,\infty}$. The proof of this result is the same as the proof of  Theorem 2.2 from \cite{gs2}, but we would like
to include it to help the reader be more familiar with the $p$-limit points.

\begin{theorem}\label{2.2} For every nonautonomous discrete dynamical system $(X,f_{1,\infty})$, we have that
$$
E(X,f_{1,\infty}) = \{ f_1^p : p \in \beta(\mathbb{N})\}
$$
and $f_1^p = p-\lim_{n \to \infty}f_1^n$, for all $p \in \mathbb{N}^*$, in the pointwise topology.
\end{theorem}

\begin{proof} The function $\psi: \mathbb{N} \to X^X$ given by
$\psi(n) = f_1^n$, for all $n \in \mathbb{N}$, is evidently continuous.
Let us consider its Stone extension $\overline{\psi}: \beta(\mathbb{N}) \to
X^X$. We know that
$$
\overline{\psi}(p) = p-lim_{n \to \infty} \psi(n) = p-lim_{n \to \infty} f_1^n = f_1^p,
$$
for all $p \in \mathbb{N}^*$. Since $\overline{\psi}[\mathbb{N}] = \psi[\mathbb{N}]$ is dense in the compact space $E(X,f_{1,\infty})$, we
obtain that $\overline{\psi}[\beta(\mathbb{N})] = E(X,f_{1,\infty})$.
\end{proof}

By the continuity of the function $\overline{\psi}:  \beta(\mathbb{N}) \to X^X$ of the previous proof, the compact space $E(X,f_{1,\infty})$ is a quotient of the compact space $\beta(\mathbb{N})$.

\medskip

Next, we shall describe the semigroup operation of the Ellis semigroup.

\medskip

Let $(X,f_{1,\infty})$ be a nonautonomous discrete
dynamical system. If $n,m \in \N$, then we  define $f_1^n \ast f_1^m := f_1^{n+m}$.
To  establish the semigroups operation for the general case we need  to
 extend  the ordinary addition on the set of
natural numbers to the whole $\beta(\mathbb{N})$ as follows:

\smallskip

For $p \in \beta(\mathbb{N})$ and $n \in \mathbb{N}$, we define $p + n := p-\lim_{m \to \infty} (m + n)$ and if $p, q \in
\beta(\mathbb{N})$, then we define $p + q :=  q - \lim_{n \to \infty}
p + n$. In terms of functions, for each $n \in \N$, we define the $n$-{\it shift}  $\lambda_n: \N \to \N$ by $\lambda_n(m) = m+n$ for each $m \in \N$. Notice
 that $\overline{\lambda_n}(p) = p+n$, for every $n \in \N$ and for every $p \in \N^*$. Moreover, $\overline{\lambda_n}$ induces  a homeomorphism from $\N^*$ to itself.  Hence, it then follows that for each $p \in \N^*$ and each $n \in \N$ there is $p-n \in \N^*$ such that  $\overline{{\lambda_n}}(p - n) = p$.
It is well known that $(\beta(\N),+)$ is a left topological semigroup (for a proof see \cite{hs}); that is, the function $q \to p+q: \beta(\N) \to \beta(\N)$ is continuous for all $p \in \beta(\N)$.

\medskip

We are ready to define the semigroup operation of $E(X,f_{1,\infty})$ for the rest of its elements. Indeed, if $p, q \in \beta(\mathbb{N})$, then we define
$$
f_1^p \ast f_1^q := f_1^{q+p}.
$$
Since the operation on $\beta(\mathbb{N})$ is associative, then  $E(X,f_{1,\infty})$ with the operation $\ast$ is a semigroup. Notice that, for a discrete dynamical system $(X,f)$, we get the Ellis semigroup $E(X,f)$  (see \cite{gs2}).
We remark that this operation of $E(X,f_{1,\infty})$ does not coincide with the composition of functions as in the Ellis semigroup of a discrete dynamical system.
For instance, let $X = [0,1]$ and let $f_n: X \to X$ be defined by $f_n(x) = \frac{1}{n+1}$, for each $x \in X$ and for each $n \in \N$. Then $f_1 \ast f_1 = f^2_n = f_2 \circ f_1 = f_2$ and $f_1 \circ f_1 = f_1$.
In what follows, we shall  also consider the subsemigroup $E(X,f_{1,\infty})^* = \{ f_1^p : p \in \N^* \}$  of the Ellis semigroup  $E(X,f_{1,\infty})$.

\medskip

Now, we establish a nice  connection between the Ellis semigroups of a nonautonomous discrete dynamical system  $(X,f_{1,\infty})$ and the Ellis semigroup of the system $(X,f_{k,\infty})$, for each positive $k \in \N$, where $f_{k,\infty}$ is the sequence of continuous functions
$(f_{n+k}: X \to X)_{n \in \N }$. First, we need to prove two lemmas.

\begin{lemma}\label{fun} Let $\sigma: \N \to \N$ be an injective function and $(x_n)_{n \in \N}$ a sequence in a compact space $X$. If $q \in \N^*$ and  $\overline{\sigma}(q) = p$, then
$$
p-\lim_{n \to \infty}x_n =  q-\lim_{l \to \infty}x_{\sigma(l)}.
$$
In particular, we have that  if $p \in \N^*$ and $k \in \N \setminus \{0\}$, then
$$
p-\lim_{n \to \infty}x_n = (p-k)-\lim_{i \to \infty}x_{k+i}.
$$
\end{lemma}

\begin{proof} Put $x = p-\lim_{n \to \infty}x_n$.  For every $V \in \mathcal{N}(x)$,  we know that
$$
 \{ n \in \N : x_n \in V\} \in p \Leftrightarrow \{ l \in \N : x_{\sigma(l)} \in V \} \in q.
$$
Hence, we obtain that $x = q-\lim_{l \to \infty}x_{\sigma(l)}$.
\end{proof}

\begin{lemma}\label{compo} Let $(X,f_{1,\infty})$ be a nonautonomous discrete dynamical systems. Then, for every $p \in \N^*$ and for every  $k \in \N$ with $k \geq 2$, we have that
$$
f_1^p = f_k^{p-k} \circ f^{k-1}_1.
$$
\end{lemma}

\begin{proof}  Let $p \in \N^*$ and $k \in \N \setminus \{0, 1\}$. Then, by definition and Lemma \ref{fun}, we obtain that
$$
f_1^p(x) = p-\lim_{n \to \infty}f_1^n(x) =  p-\lim_{n \to \infty}  f_n \circ f_{n-1}  \circ....... \circ f_2 \circ f_1(x)
$$
$$
= (p-k) -\lim_{i \to \infty}  f_{k+i} \circ..... \circ f_{k} \circ f_{k-1}  \circ....... \circ f_2 \circ f_1(x)
$$
$$
= (p-k) -\lim_{i \to \infty}  f_{k+i} \circ..... \circ f_{k}\big(f_{k-1}  \circ....... \circ f_2 \circ f_1(x)\big)
$$
$$
= f_k^{p-k}\big(f_{k-1}  \circ....... \circ f_2 \circ f_1(x)\big),
$$
for all $x \in X$.
This shows that  $f_1^p = f_k^{p-k} \circ f^{k-1}_1$.
\end{proof}

\begin{theorem}\label{sur} Let $(X,f_{1,\infty})$ be a nonautonomous discrete dynamical systems. Then, the space
$E(X,f_{1,\infty})^*$ is a continuous image of $E(X,f_{k,\infty})^*$ for each  $k \in \N$ with $k \geq 2$.
\end{theorem}

\begin{proof}
Let $(X,f_{1,\infty})$ be a nonautonomous discrete dynamical systems and fix $k \in \N \setminus \{0, 1\}$. By Lemma \ref{compo},
 we obtain that
$$
f_k^{p} \circ f^{k-1}_1 = f_k^{(p+k)-k} \circ f^{k-1}_1 = f_1^{p+k},
$$
for every $p \in \N^*$. Since the composition on the right of the set $X^X$ is continuous with the pointwise topology, it follows that the function
$$
f_k^p \to f_k^{p} \circ f^{k-1}_1 : E(X,f_{k,\infty})^* \to E(X,f_{1,\infty})^*
$$
is  continuous. Moreover, since the function $n \mapsto n+k: \N \to \N$ induces a bijection of $\N^*$, given  $q \in \N^*$ there is $p \in \N^*$ such that $q = p+k$. Hence,
the function $f_k^p \to f_k^{p} \circ f^{k-1}_1$ is a surjection.
\end{proof}

Theorem \ref{sur} gives us a sequence of continuous surjections:
$$
...... \to E(X,f_{k,\infty})^* \to ..... \to E(X,f_{2,\infty})^* \to  E(X,f_{1,\infty})^*,
$$
for every nonautonomous discrete dynamical system $(X,f_{1,\infty})$.

\medskip

 To state some properties of the Ellis semigroup of a  periodic nonautonomous discrete dynamical system we need the next notation.

\smallskip

For each $1 < k \in \N$ and $i < k$, we let $C_i = \{ n \in \N : n \equiv  i \ \text{mod}(k) \  \}$.

\begin{theorem}\label{periodic} Let  $(X,f_{1,\infty})$  be a nonautonomous discrete dynamical systems such that
$$
f_{1,\infty} = (f_1,f_2,....,f_k,f_1,f_2,....,f_k,f_1,f_2,....,f_k,....)
$$
where $1 < k \in \N$ and $f_i: X \to X$ is a continuous functions for each $1 \leq i \leq k$. If $p \in C_j^*$ for some $ j < k$, then
$f_1^p \in f_j \circ .... \circ f_2 \circ f_1  \circ E(X,f_{k} \circ ......\circ f_2  \circ f_1)$.
\end{theorem}

\begin{proof} First, notice that if $n = kl + j$ for some  $l \in \N$ and  $j < k$, then we obtain that
$$
f_{1}^n = f_{1}^{kl+j}  =  f_j \circ .... \circ f_2 \circ f_1 \circ (f_k \circ  .... \circ f_2 \circ f_1)^l.
$$
Fix $j < k$ and  $p \in C_j^*$. Consider the function $\sigma_{j}: \N \to C_j$ defined by $\sigma_j(l) = kl + j$, for all $l \in \N$, and let
$q \in \N^*$ such that $\overline{\sigma_j}(q) = p$. By Lemma \ref{fun}, we have that
$$
f_1^p(x) =   p-\lim_{n \to \infty}  f_1^n(x) = q-\lim_{l \to \infty}  f_1^{\sigma_j(l)}(x) =
q-\lim_{l \to \infty}  f_1^{kl + j}(x)
$$
$$
= q-\lim_{l \to \infty}f_j \circ .... \circ f_2 \circ f_1 \big( (f_k \circ  .... \circ f_2 \circ f_1)^l(x) \big)
$$
$$
= f_j \circ .... \circ f_2 \circ f_1 \big(q-\lim_{l \to \infty} (f_k \circ  .... \circ f_2 \circ f_1)^l(x) \big)
$$
$$
= f_j \circ .... \circ f_2 \circ f_1 \big((f_k \circ  .... \circ f_2 \circ f_1)^q(x)\big)
$$
for each $x \in X$. Thus, we must have that
$$
f_1^p = f_j \circ .... \circ f_2 \circ f_1 \circ \big((f_k \circ  .... \circ f_2 \circ f_1)^q\big) \in f_j \circ .... \circ f_2 \circ f_1  \circ E(X,f_{k} \circ ......\circ f_2  \circ f_1).
$$
\end{proof}

 The  nonautonomous discrete dynamical systems $(X,f_{1,\infty})$ satisfying the conditions of Theorem \ref{periodic} are called {\it periodic} and, as we have shown, its   Ellis semigroup can be described by means of the Ellis semigroups of a certain discrete dynamical system.

\medskip

Following the paper \cite{shi} we say that a  nonautonomous discrete dynamical system $(X,\hat{f}_{1,\infty})$ is {\it induced} by the  nonautonomous discrete dynamical system $(X,f_{1,\infty})$ if there is a strictly increasing sequence $(k_n)_{n \in \N}$ of natural numbers such that $\hat{f}_1 := f_{k_1} \circ..... \circ f_1$ and $\hat{f}_n := f_{k_n} \circ..... \circ f_{k_{n-1}+1}$ for each positive $n \in \N$ bigger than $1$. If $(k_n)_{n \in \N}$ is a strictly increasing sequence of positive integers, then we consider the function $\gamma: \N \to \N$ defined by $\gamma(n) = k_n$ for every $n \in \N$. This function will helps us to connect the Ellis semigroups of a nonautonomous discrete dynamical system and one induced by it as follows.

\begin{theorem}\label{formu} If $(X,\hat{f}_{1,\infty})$ is  a  nonautonomous discrete dynamical system {\it induced} by the  nonautonomous discrete dynamical system $(X,f_{1,\infty})$, then
$$
\hat{f}_1^q = f_1^p,
$$
provided that $p, q \in \N^*$ and $\overline{\gamma}(q) = p$, where $\gamma: \N \to \N$ is the function induced by the sequence $(k_n)_{n \in \N}$. In particular,  we have that $E(X,\hat{f}_{1,\infty}) \subseteq E(X,f_{1,\infty})$.
\end{theorem}

\begin{proof}  We know that
$$
\hat{f}_1^n = f_{k_n} \circ..... \circ f_1 = f_1^{k_n},
$$
for every $n \in \N$. Suppose that $p, q \in \N^*$ satisfy that $\overline{\gamma}(q) = p$ and fix $x \in X$.
If $U \in \mathcal{N}(\hat{f}_1^q(x))$, then we have that $\{ n \in \N : \hat{f}_1^n(x) = f_1^{k_n}(x) \in U \} \in q$ which implies that
$\{ k_n \in \N : f_1^{k_n}(x) \in U \} \in p$. That is, $f^p_1(x) \in cl_X(U)$ and so  $\hat{f}_1^q(x) = f^p_1(x)$.
\end{proof}

In the next theorem, we give a necessary condition which implies that $E(X,\hat{f}_{1,\infty})$ is a subsemigroup of  $E(X,f_{1,\infty})$.

\begin{lemma}\label{sum} Let $\gamma: \N \to \N$ be a one-to-one function.
If $\gamma(n+m) = \gamma(n) + \gamma(m)$ for each $n,m \in \N$, then
$$
\overline{\gamma}(p+q) = \overline{\gamma}(p) + \overline{\gamma}(q),
$$
for each $p, q \in \beta(\N)$.
\end{lemma}

\begin{proof} Let $p, q \in \beta(\N)$. Then, by the left continuity of the addition  on $\beta(\N)$, we have that
$$
\overline{\gamma}(p+q) = \overline{\gamma}(q-\lim_{n \to \infty}p+n) = q-\lim_{n \to \infty}\overline{\gamma}(p+n) =
$$
$$
q-\lim_{n \to \infty}\overline{\gamma}(p-\lim_{m \to \infty}(n+m)) = q-\lim_{n \to \infty}\big(p-\lim_{m \to \infty}\gamma(n+m)\big)
$$
$$
= q-\lim_{n \to \infty}\big(p-\lim_{m \to \infty}(\gamma(n) + \gamma(m))\big)
$$
$$
= q-\lim_{n \to \infty}\big(\gamma(n) + (p-\lim_{m \to \infty} \gamma(m))\big)
$$
$$
=   \big(p-\lim_{m \to \infty} \gamma(m)\big) + \big(q-\lim_{n \to \infty}\gamma(n)\big)
$$
$$
= \overline{\gamma}(p) + \overline{\gamma}(q).
$$
\end{proof}

\begin{theorem} $(X,\hat{f}_{1,\infty})$ is  a  nonautonomous discrete dynamical system {\it induced} by the  nonautonomous discrete dynamical system $(X,f_{1,\infty})$ via a strictly increasing sequence $(k_n)_{n \in \N}$ of positive integers. If $k_n + k_m = k_{n+m}$ for all $n, m \in \N$, then
$E(X,\hat{f}_{1,\infty})$ is a subsemigroup of  $E(X,f_{1,\infty})$.
\end{theorem}

\begin{proof} Consider the function $\gamma: \N \to \N$ defined by $\gamma(n) = k_n$ for each $n \in \N$. Fix $p, q \in \beta(\N)$ and put
$\overline{\gamma}(p) = p'$ and $\overline{\gamma}(q) = q'$.  By  Lemma \ref{sum}, we know that
$\overline{\gamma}(p+q) = \overline{\gamma}(p) + \overline{\gamma}(q) = p' + q'$ and so, by Theorem \ref{formu}, we obtain that
$$
\hat{f}_1^{p+q} = f_1^{p'+q'} = f_1^{q'} \ast f_1^{p'} = \hat{f}_1^{q} \ast \hat{f}_1^{p}.
$$
\end{proof}

\section{Some applications of the Ellis semigroup}

In this section, we shall mainly consider those nonautonomous discrete dynamical systems $(X,f_{1,\infty})$ for which
the sequence $(f_n)_{n \in \N}$ converges uniformly to a function $\phi: X \to X$. For our purposes we need the following general notions and some preliminary results.

\begin{definition} Let $X$ be a compact metric space and let $(f_n)_{n \in \N}$ and $(g_n)_{n \in \N}$ be
two sequences of conti\-nuous functions from $X$ to $X$.
We say that  $(f_n)_{n \in \N}$ and $(g_n)_{n \in \N}$ are  {\it asymptotic} at $x \in X$ if for every $\epsilon > 0$ there is $k \in \N$ such that
$$
d(f_n(x),g_n(x)) < \epsilon \   \text{for every}  \ n \in \N \ \text{with} \ n > k.
$$
$(f_n)_{n \in \N}$ and $(g_n)_{n \in \N}$ are  said to be {\it asymptotic} if they are asymptotic at every point of $X$.
The sequences  $(f_n)_{n \in \N}$ and $(g_n)_{n \in \N}$ are called {\it uniformly asymptotic} if for every $\epsilon > 0$ there is $k \in \N$ such that
$$
d(f_n(x),g_n(x)) < \epsilon \ \text{for every} \ x \in X \ \text{and } \ n \in \N \ \text{with} \ n > k.
$$
\end{definition}

By using free ultrafilter convergence we may generalize the previous notions as follows:

\begin{definition} Let $X$ be a compact metric space and let $(f_n)_{n \in \N}$ and $(g_n)_{n \in \N}$ be
two sequences of conti\-nuous functions from $X$ to $X$.
For $p \in \N^*$, we say that
$(f_n)_{n \in \N}$ and $(g_n)_{n \in \N}$  are $p$-{\it asymptotic} at $x$ if for every $\epsilon > 0$ we have that
$$
\{ n \in \N : d(f_n(x),g_n(x)) < \epsilon \}  \in p.
$$
$(f_n)_{n \in \N}$ and $(g_n)_{n \in \N}$ are  said to be $p$-{\it asymptotic} if they are $p$-asymptotic at every point of $X$.
The sequences $(f_n)_{n \in \N}$ and $(g_n)_{n \in \N}$  are said to be $p$-{\it uniformly asymptotic}
if for every $\epsilon > 0$ we have that
$$
\{ n \in \N : \forall x \in X \big(d(f_n(x),g_n(x)) < \epsilon \big)   \}  \in p.
$$
\end{definition}

It is evident that  the asymptotic property  implies the $p$-asymptotic property, for each $p \in \N^*$. For a given $p \in \N^*$, it is easy to construct an example of two $p$-asymptotic  sequences of continuous functions that are not asymptotic.

\begin{lemma}\label{aprox} Let $X$ be a compact metric space and let $p \in \N^*$.
Let $(f_n)_{n \in \N}$ and $(g_n)_{n \in \N}$ be two sequences of continuous functions from $X$ to $X$. If $(f_n)_{n \in \N}$ and $(g_n)_{n \in \N}$ are $p$-asymptotic at $x \in X$, then
$$
p-\lim_{n \to \infty}f_n(x) = p-\lim_{n \to \infty}g_n(x).
$$
\end{lemma}

\begin{proof} Fix $\epsilon > 0$. Put $y = p-\lim_{n \to \infty}f_n(x)$ and $z = p-\lim_{n \to \infty}g_n(x)$. We know that
$$
\{ n \in \N : d(f_n(x), y) < \epsilon  \}  \cap \{ n \in \N : d(g_n(x),z) < \epsilon  \} \in p \ \text{and}
$$
$$
\{ n \in \N : d(f_n(x),g_n(x)) < \epsilon \}  \in p.
$$
Hence, it is possible to find $m \in \N$ so that
$$
d(y,z) \leq  d(f_m(x),y) + d(f_m(x),g_m(x)) + d(g_m(x),z) < 3\epsilon.
$$
Since $\epsilon$ was taken arbitrarily, we conclude that $y = z$.
\end{proof}

\begin{lemma}\label{uniap} Let $(X,f_{1,\infty})$ be a nonautonomous discrete dynamical system. If the sequence $(f_n)_{n \in \N}$ converges uniformly to a function $\phi: X \to X$, then  the sequences  $(\phi \circ f_1^n)_{n \in \N}$ and $(f_1^{n+1})_{n \in \N}$ are uniformly asymptotic.
\end{lemma}

\begin{proof}  Let $\epsilon > 0$. By assumption,  there is $k \in \N$ such that
$d(\phi(x),f_n(x)) < \epsilon$,  for every $x \in X$ and for each  $n \in \N$ bigger than $k$. This implies that
$$
d(\phi(f_1^n(x)),f_1^{n+1}(x)) = d(\phi(f_1^n(x)),f_{n+1}(f_1^n(x))) < \epsilon,
$$
for every $n \in \N$ with $n > k$ and for every $x \in X$.
\end{proof}

Let  see that in a  nonautonomous discrete dynamical systems that we are considering
 some elements of their Ellis semigroup can be somehow related one to the other.

\begin{theorem}\label{ca} Let $(X,f_{1,\infty})$ be a nonautonomous discrete dynamical system. If  $(f_n)_{n \in \N}$ converges uniformly to a function $\phi: X \to X$, then $\phi \circ f_1^p = f_1^{p+1}$ for each $p \in \N^*$.
\end{theorem}

\begin{proof} Let $p \in \N^*$ and let $x \in X$. By Lemma \ref{uniap}, we have that the sequences  $(\phi \circ f_1^n)_{n \in \N}$ and $(f_1^{n+1})_{n \in \N}$ are  asymptotic at $x$. Then, by Lemma \ref{aprox}, we obtain that
$$
p-\lim_{n \to \infty}\phi \circ f_1^n(x) = p-\lim_{n \to \infty}f_1^{n+1}(x).
$$
 Hence,
$$
\phi \circ f_1^p(x) = \phi\big(p-\lim_{n \to \infty}  f_1^n(x)\big) = p-\lim_{n \to \infty}\phi \circ f_1^n(x) = p-\lim_{n \to \infty}f_1^{n+1}(x)
$$
$$
= (p+1)-\lim_{n \to \infty}f_1^{n}(x) = f_1^{p+1}(x).
$$
\end{proof}

The condition $p \in \N^*$ of the previous theorem is necessary: Indeed,
for each $n \in \N$,   $f_{n}: [0,1] \to [0,1]$ will denote the constant function with value $\frac{1}{n+1}$. It is evident that the sequence  $(f_n)$ converges uniformly to the constant function $\phi: [0,1] \to [0,1]$ with value $0$ and $\phi^m(f_1^n (x)) \neq f_1^{n+m}(x)$ for all $x \in X$ and for all $m, n \in \N$.

\medskip

Next, we give a new proof of a result of R. Kempf \cite{kempf} which uses elements of  $E(X,f_{1,\infty})$. This proof  is quite different from the one given in \cite{ca}.

\medskip

For a nonautonomous discrete dynamical system $(X,f_{1,\infty})$, recall that the  $\omega$-{\it limit set} of a point $x \in X$ is the set
$$
\omega(x,f_{1,\infty}) :=  \bigcap_{ n \in \N }cl(\{ f_{1}^m(x) : n \leq m \in \N  \}).
$$
This generalization of $\omega$-limit set was first considered in the paper \cite{kempf}.
In terms of $p$-limit points we have that
$$
\omega(x,f_{1,\infty}) = \{ y \in X : \exists   p \in \N^* ( y = p-\lim_{n \to \infty}f_1^n(x)) \},
$$
for each $x \in X$.

\begin{theorem}{\bf [R. Kempf]} Let $(X,f_{1,\infty})$ be a nonautonomous discrete dynamical system. If $(f_n)_{n \in \N}$ converges uniformly to a function $\phi: X \to X$, then
$$
\phi\big(\omega(x,f_{1,\infty})\big) = \omega(x,f_{1,\infty}),
$$
for every $x \in X$.
\end{theorem}

\begin{proof}  According to Theorem \ref{ca}, for every  $x \in X$ and $p \in \N^*$ we have that
$$
\phi \big(f_1^{p}(x)\big) = \phi \circ f_1^{p}(x) = f_1^{p+1}(x),
$$
and
$$
\phi \big(f_1^{p-1}(x)\big) = \phi \circ f_1^{p-1}(x) = f_1^{(p-1)+1}(x) = f_1^p(x).
$$
Thus, we obtain that  $\phi\big(\omega(x,f_{1,\infty})\big) = \omega(x,f_{1,\infty})$.
\end{proof}

We already have mentioned that  $E(X,f_{1,\infty})$ could contain  discontinuous functions. In the next result, we give a necessary
condition that guarantees the continuity of all function of the Ellis semigroup.

\begin{theorem}\label{itecon} Let $(X,f_{1,\infty})$ be a nonautonomous discrete dynamical system. If  either
\begin{enumerate}
\item the family $\{  f_1^n : n \in \N \}$ is equicontinuous; or

\item $(f_n)_{n \in \N}$ converges uniformly to a function $\phi: X \to X$ and the family $\{ \phi \circ f_1^n : n \in \N \}$ is equi\-con\-ti\-nuous,
\end{enumerate}
then the function $f_1^p: X \to X$ is continuous for each $p \in \N^*$.
\end{theorem}

\begin{proof} We only prove the theorem when the sequence $(f_n)_{n \in \N}$ converges uniformly to a function $\phi: X \to X$ and the family $\{ \phi \circ f_1^n : n \in \N \}$ is equicontinuous. Let $\epsilon > 0$ and choose $\delta > 0$ so that if $x, y \in X$ and $d(x,y) < \delta$, then
$$
d(\phi \circ f_1^{n}(x),\phi \circ f_1^{n}(y)) < \frac{\epsilon}{5},
$$
for every $n \in \N$. Let $x, y \in X$ be such that $d(x,y) < \delta$. By Lemma \ref{uniap} there is $k \in \N$ such that
$$
d(\phi \circ f_1^{n}(x),f_1^{n+1}(x)) < \frac{\epsilon}{5},
$$
for each $n \in \N$ with $n > k$. Fix  $p \in \N^*$. Then choose $m \in \N$ so that $m > k$ and
$$
d(f_1^p(x),f_1^p(y)) \leq  d(f_1^p(x),f_1^{m+1}(x)) + d(f_1^{m+1}(x),f_1^{m+1}(y)) + d(f_1^{m+1}(y),f_1^p(y))
$$
$$
\leq d(f_1^p(x),f_1^{m+1}(x))  +  d(\phi \circ f_1^{m}(x),f_1^{m+1}(x)) + d(\phi \circ f_1^{m}(x),\phi \circ f_1^{m}(y))
$$
$$
+ d(\phi \circ f_1^{m}(y),f_1^{m+1}(y)) + d(f_1^{m+1}(y),f_1^p(y)) \leq
$$
$$
\frac{\epsilon}{5} + \frac{\epsilon}{5} + \frac{\epsilon}{5} + \frac{\epsilon}{5} + \frac{\epsilon}{5} =  \epsilon.
$$
This shows the continuity of $f_1^p$.
\end{proof}

We shall see, in the next theorem,  that there is a nice relationship between the Ellis semigroup $E(X,\phi)$ and $E(X,f_{1,\infty})^*$ when the sequence of functions $(f_n)_{n \in \N}$ converges uniformly to a function $\phi: X \to X$.

\begin{lemma}\label{suma} Let $(X,f_{1,\infty})$ be a nonautonomous discrete dynamical system. Then,
$$
f_1^{p+q}(x) = q-\lim_{n \to \infty}f_1^{p+n}(x),
$$
for every $p, q \in \N^*$ and  $x \in X$.
\end{lemma}
\begin{proof} Fix $p, q \in \N^*$ and  $x \in X$. We claim that
$$
f_1^{p+q}(x) = (p+q)-\lim_{n \to \infty}f_1^n(x) = q-\lim_{n \to \infty}f_1^{p+n}(x).
$$
In fact, put $y = (p+q)-\lim_{n \to \infty}f_1^n(x)$ and $z = q-\lim_{n \to \infty}f_1^{p+n}(x)$.
If $V \in \mathcal{N}(y)$, then $A = \{ n \in \N :  f_1^n(x) \in V \} \in p+q$ which implies that
$B = \{ n \in \N : p+n \in A^* \} \in q$. Fix $n \in B$. As  $p+n \in A^*$, we must have that  $\{ m \in \N : m + n \in A \} \in p$ and hence
$\{ m + n  \in \N : m + n \in A \} \in p + n$. Thus, $\{ k \in \N : f_1^k(x) \in V \} \in p+n$. So, $f_1^{p+n}(x) \in \overline{V}$ for each $n \in B$. That is,
$\{ n \in \N : f_1^{p+n}(x) \in \overline{V}\} \in q$. Therefore, $y = z$ because of  the uniqueness of the limit points with respect to an ultrafilter.
\end{proof}

\begin{theorem}\label{accion} Let $(X,f_{1,\infty})$ be a nonautonomous discrete dynamical system such that $(f_n)_{n \in \N}$ converges uniformly to a function $\phi: X \to X$. Then $\phi^q \circ f_1^p = f_1^{p + q}$ for every $p \in \N^*$ and for every $q \in \beta(\N)$. In particular, the Ellis semigroup $E(X,\phi)$ acts on the  semigroup $E(X,f_{1,\infty})^*$.
\end{theorem}

\begin{proof} Fix $p \in \N^*$ and $q \in \beta(\N)$. In virtue of Theorem \ref{ca}, we know that $\phi \circ f_1^p = f_1^{p+1}$.  Hence,
we obtain that $\phi^n \circ f_1^p = f_1^{p+n}$ for each $n \in \N$. So, by Lemma \ref{suma}, we have that
$$
\phi^q \circ f_1^p(x) =  q -\lim_{n \to \infty}\phi^n(f_1^p(x)) = q -\lim_{n \to \infty}f_1^{p+n}(x)
$$
$$
= (p+q)-\lim_{n \to \infty}f_1^n(x) = f_1^{p+q}(x)
$$
for each $x \in X$.
\end{proof}

Now, we shall give an example of   a nonautonomous discrete dynamical system  $(X,f_{1,\infty})$ such that
$(f_n)_{n \in \N}$ converges uniformly to a function $\phi: X \to X$ and $\phi^p \circ f_1 \notin  E(X,f_{1,\infty})$  for any $p \in \beta(\N)$.
This  witnesses that the Ellis semigroup $E(X,\phi)$ does not necessarily act on the entire semigroup $E(X,f_{1,\infty})$.

\begin{example}
Let us consider the unit interval $[0,1]$. We define $\phi: [0,1] \to [0,1]$ by
$$
\phi(x)= \begin{cases} x  &  \ \ \text{if} \ \ x \in [0,\frac{1}{2}]  \\
\frac{1}{2}  &  \ \ \text{if} \ \ x \in (\frac{1}{2},1] ,
 \end{cases}
$$
 and
$$
f_1(x)= \begin{cases} \frac{2}{3}x  &  \ \ \text{if} \ \ x \in [0,\frac{1}{2}]  \\
\frac{1}{3}(x + \frac{1}{2})  &  \ \ \text{if} \ \ x \in (\frac{1}{2},1] .
 \end{cases}
$$
The  graph of the function  $f_1$ consists of two closed segments of $\mathbb{R}^2$ that connect the point $(0,0)$ with $(\frac{1}{2},\frac{1}{3})$ and
the point $(\frac{1}{2},\frac{1}{3})$ with $(1,\frac{1}{2})$.
For each $1 < n \in \N$ we define
$$
f_n(x)= \begin{cases} \frac{n}{n+1}x  &  \ \ \text{if} \ \ x \in [0,\frac{1}{n})  \\
(x - \frac{1}{2})(\frac{n(2-n-1)}{(n+1)(2-n)}) + \frac{1}{2}  &  \ \ \text{if} \ \ x \in [\frac{1}{n},\frac{1}{2}) \\
\frac{1}{2}  &  \ \ \text{if} \ \ x \in [\frac{1}{2},1].
 \end{cases}
$$
 In this case, the graph of the function  $f_n$ consists of three closed segments of $\mathbb{R}^2$ which are the one connecting the points $(0,0)$ and $(\frac{1}{n},\frac{1}{n+1})$, the one connecting the points   $(\frac{1}{n},\frac{1}{n+1})$ and $(\frac{1}{2},\frac{1}{2})$, and the one connecting the points $(\frac{1}{2},\frac{1}{2})$ and $(1,\frac{1}{2})$.
Clearly, the sequence $(f_n)_{n \in \N}$ converges uniformly to the function $\phi$.
Also, we have that $f_1(1) = \frac{1}{2}$,  $f_n(\frac{1}{n}) = \frac{1}{n+1} = f_1^n(1)$, for every $2 \leq n \in \N$, and
 $\phi^p = \phi$, for each $p \in \beta(\N)$. Thus,
 $\phi^p(\frac{1}{2}) = \frac{1}{2}$, for each  $p \in \beta(\N)$, and hence $\phi^p(f_1(1)) = \phi^p(1) = \frac{1}{2}$,  for each $p \in \beta(\N)$.
 As $f_1^q(1) = 0$  for each $q \in \N^*$, we obtain that  $\phi^p \circ f_1 \neq f_1^q$ for every $p, q \in \beta(\N)$.
\end{example}

Given  a nonautonomous discrete dynamical system  $(X,f_{1,\infty})$, for each $p \in \N^*$ we define
$$
E_p(X,f_{1,\infty}) := \{ f_1^{p+q} : q \in \beta(\N) \}.
$$
Notice that $E_p(X,f_{1,\infty})$ is a subsemigroup of $E(X,f_{1,\infty})$ for every $p \in  \N^*$.

\begin{theorem}\label{image} Let $(X,f_{1,\infty})$ be a nonautonomous discrete dynamical system that $(f_n)_{n \in \N}$ converges uniformly to a function $\phi: X \to X$.  Then $E_p(X,f_{1,\infty})$ is a continuous image of $E(X,\phi)$ for each $p \in \N^*$. As a consequence, $E_p(X,f_{1,\infty})$  is a quotient space of  the space $E(X,\phi)$ for every $p \in \N^*$.
\end{theorem}

\begin{proof} Fix $p \in \N^*$. Define $\Psi_p: E(X,\phi) \to E(X,f_{1,\infty})$ by $\Psi_p(\phi^q) = f_1^{p + q}$ for each $q \in \beta(\N)$.
This function $\Psi_p$ is well defined. Indeed,   if $\phi^s = \phi^t$ for some $s, t \in \beta(\N)$,  by Theorem \ref{accion}, then
 $\phi^s \circ f_1^p = f_1^{p + s} = f_1^{p + t} = \phi^t \circ f_1^p$. Let us show that $\Psi_p$ is continuous. Consider the the open set $V = \bigcap_{i \leq l}[x_i,V_i]$ where $x_i \in X$ and $V_i \subseteq X$ is a nonempty open set for each $i \leq l$.   If $\phi^q \in \bigcap_{i \leq l}[f^p(x_i),V_i]$ for some $q \in \beta(\N)$, then
 $\phi^q(f_1^p(x_i)) = f_1^{p + q}(x_i) \in V_i$ for all $i \leq l$. This shows that $\Psi_p$ is continuous.
\end{proof}

The next theorem provides  a very interesting information concerning fixed points of the function $\phi$ when the sequence $(f_n)_{n \in \N}$ converges uniformly to  $\phi$.

\begin{theorem} Let $(X,f_{1,\infty})$ be a nonautonomous discrete dynamical system such that  $(f_n)_{n \in \N}$ converges uniformly to a function $\phi: X \to X$. Then,  the function $\phi: X \to X$ has a fixed point in
$\omega(x,f_{1,\infty})$ for every $x \in X$.
\end{theorem}

\begin{proof} Fix $x \in X$ and consider the set $\{ d(f_1^p(x),f_1^{p+1}(x)) : p \in \N^* \}$. Since the function $p \mapsto (f_1^p(x),f_1^{p+1}(x)): \beta(\N) \to X \times X$ is continuous\footnote{It is the Stone extension of the continuous function $n \mapsto (f_1^n(x),f_1^{n+1}(x)): \N \to X \times X$.} and $\N^*$ is compact, then there is $q \in \N^*$ such that
$$
d(f_1^q(x),f_1^{q+1}(x))  = \inf\{ d(f_1^p(x),f_1^{p+1}(x)) : p \in \N^* \}.
$$
If $d(f_1^q(x),f_1^{q+1}(x)) = 0$, then it follows from Theorem \ref{accion} that  $\phi (f_1^q(x)) = f_1^{q + 1}(x) = f_1^q(x)$ and hence $f_1^q(x)$ is a fixed point of $\phi$. Suppose that $0 < \epsilon = d(f_1^q(x),f_1^{q+1}(x))$. Choose $k \in \N$ so that
$$
d(\phi(f_1^n(x)),f_1^{n+1}(x)) = d(\phi(f_1^n(x)),f_{n+1}(f_1^n(x))) < \frac{\epsilon}{2},
$$
for every $n \in \N$ with $n > k$.  On the other hand, we know that $\phi(f_1^q(x)) = q-\lim_{n \to \infty}\phi(f_1^{n}(x))$ and
$$
f_1^{q+1}(x) = (q+1)-\lim_{n \to \infty}f_1^{n}(x) =  q -\lim_{n \to \infty}f_1^{n+1}(x).
$$
Hence, we obtain that  $d(f_1^q(x),f_1^{q+1}(x)) \leq  \frac{\epsilon}{2}$ which is a contradiction.
\end{proof}

The following two corollaries are direct consequence of Theorem \ref{accion}.

\begin{corollary} Let $(X,f_{1,\infty})$ be a nonautonomous discrete dynamical system such that  $(f_n)_{n \in \N}$ converges uniformly to a function $\phi: X \to X$. If $p \in \N^*$ is an idempotent\footnote{That is, $p+p = p$.}, then $\phi^p \circ f_1^{q+p} = f_1^{q+ p+p} = f_1^{q+p}$ for all $q \in \N^*$. That is,
$f_1^{q+p}(x)$ is a fixed point of $\phi^p$ for every $q \in \N^*$ and for every $x \in X$.
\end{corollary}

The next results should be compared with Proposition 2.1 from \cite{ca11}.

\begin{corollary} Let $(X,f_{1,\infty})$ be a nonautonomous discrete dynamical system such that  $(f_n)_{n \in \N}$ converges uniformly to a function $\phi: X \to X$, let  $p \in \N^*$ and let $x \in X$. Then, $f_1^p(x)$ is a periodic point of $\phi$ iff there is $n \in \N$ such that
$f_1^p(x) = f_1^{p+ n}(x)$.
\end{corollary}

Our last task is to  prove that the Ellis semigroup of two  topologically conjugate nonautonomous discrete dynamical systems are topologically isomorphic.
Let us remind the definition of topologically conjugate which was introduced in \cite{cl}.

\begin{definition}  Two  nonautonomous discrete dynamical systems $(X,f_{1,\infty})$ and $(Y,g_{1,\infty})$ are called {\it topologically semi-conjugate}
if there is a surjective continuous function $h: X \to Y$ such that $g_n(h(x)) = h(f_n(x))$ for every $x \in X$ and for every $n \in \N$. If the function $h$ is a homeomorphism, then we say that  they are {\it topologically conjugate}.
\end{definition}

\begin{theorem} Let $(X,f_{1,\infty})$ and $(Y,g_{1,\infty})$  be two nonautonomous discrete dynamical systems topologically semi-conjugate
via the surjective continuous function $h: X \to Y$. Then, $E(Y,g_{1,\infty})$ is a continuous image of $E(X,f_{1,\infty})$. If $h$ is a homeomorphism, then
$E(Y,g_{1,\infty})$ and $E(X,f_{1,\infty})$ are topologically isomorphic.
\end{theorem}

\begin{proof} Sine $g_n \circ h = h \circ f_n$ for every $n \in \N$, we must have that
$g_1^n \circ h = h \circ f_1^n$ for every $n \in \N$. Hence, we obtain that
 $g_1^p \circ h = h \circ f_1^p$ for every $p \in \beta(\N)$.
We shall prove that the  function $H: E(X,f_{1,\infty}) \to E(Y,g_{1,\infty})$ defined by $H(f_1^p) = g_1^p$, for each $p \in \beta(\N)$, is  well-defined and continuous. To prove that this function is well-defined assume that $f_1^p = f_1^q$ for some $p, q \in \beta(\N)$. Then, $g_1^p \circ h = h \circ f_1^p = h \circ f_1^q = g_1^q \circ h$ and since $h$ is onto, we must have that $g_1^p = g_1^q$. Now, we shall show that the function $H$ is continuous. Consider the basic open subset $V = \bigcap_{i \leq l}[y_i,V_i]$ where $y_i \in Y$ and $V_i$ is a nonempty open subset of $Y$, for every $i \leq l$. For each $i \leq l$, set $U_i = h^{-1}(V_i)$ and choose $x_i \in X$ so that $h(x_i) = y_i$. Put $U = \bigcap_{i \leq l}[x_i,U_i]$ and suppose that $f_1^p \in U$ for some $p \in \beta(\N)$. Then, $f_1^p(x_i) \in U_i$ and so $h(f_1^p(x_i)) = g_1^p(h(x_i)) = g_1^p(y_i) \in V_i$, for each $i \leq l$. That is, $H(f_1^p) = g_1^p \in V$. This shows the continuity of the function $H$. It is clear that if $h$ is a homeomorphism, then
$H$ is a topological isomorphism.
\end{proof}

 Let $(X,f_{1,\infty}$ nonautonomous discrete dynamical system, where  $f_{1,\infty} = (f_n: X \to X)_{n \in \N}$.
The orbit $\mathcal{O}_{f_{1,\infty}}(x) := \{x, f_1^1(x), f_1^2(x),....., f_1^n(x),.....\}$ of a point $x \in X$
can be  also described by the difference equation:
$x_1 = x$  and $x_{n+1} = f_n(x_n) = f_1^n(x)$ for each positive  $n \in \N$.
We may generalize the notion of nonautonomous difference equation  as follows:

For $x  \in X$ and $p \in \beta(\N)$, we define $x_1 := x$ and  $x_{p} := p-\lim_{n \to \infty}f_{n-1}(x_{n-1}) =  p-\lim_{n \to \infty}x_{n}$.

 \begin{theorem}\label{eq} Let $(X,f_{1,\infty})$ be a nonautonomous discrete dynamical system and $x \in X$.  Consider the nonautonomous  difference equation
$ x_1 = x,$ and $x_{n+1} = f_n(x_n)$ for each $n \in \mathbb{N}$. Then   we have that
 $$
 x_p   = f_1^{p-1}(x),
 $$
 for all $p \in \beta(\N)$. Besides, we have that $cl_X(\{ x_n : n \in \N\}) = \{ x_p : p \in \beta(\N)\}$.
\end{theorem}

 \begin{proof} According to Lemma \ref{fun}, we have that
 $$
 x_p = p-\lim_{n \to \infty}f_{n-1}(x_{n-1}) = p-\lim_{n \to \infty}x_{n} = (p-1)-\lim_{n \to \infty}x_{n+1} = f_1^{p-1}(x).
 $$
 The function $n \mapsto x_n : \N \to X$ extends continuously to the function $p \mapsto x_p : \beta(\N) \to X$ and hence we obtain that
 $cl_X(\{ x_n : n \in \N\}) = \{ x_p : p \in \beta(\N)\}$.
 \end{proof}

Assume that $(X,f_{1,\infty})$ is a nonautonomous discrete dynamical system such that  $(f_n)_{n \in \N}$ converges uniformly to a function $\phi: X \to X$.
According to Theorem \ref{ca}, we have that $\phi(x_{p+n}) = x_{p+n+1}$ for each $p \in \N^*$, $n \in \N$ and $x \in $X.

\bibliographystyle{amsplain}

\end{document}